\documentclass[11pt]{article}
\usepackage[a4paper]{anysize}\marginsize{3.5cm}{3.5cm}{1.3cm}{2cm}
\pdfpagewidth=\paperwidth \pdfpageheight=\paperheight
\usepackage{amsfonts,amssymb,amsthm,amsmath,eucal}
\usepackage{bbm}
\usepackage{tikz} %Used for drawing picture in LaTeX
\usepackage{mathrsfs}

\theoremstyle{plain}
\newtheorem{thm}{Theorem}[section]
\newtheorem{theorem}[thm]{Theorem}

\newtheorem{lemma}[thm]{Lemma}

\newtheorem{corollary}[thm]{Corollary}
\newtheorem{conjecture}[thm]{Conjecture}

\theoremstyle{definition}
\newtheorem{definition}[thm]{Definition}
\newtheorem{remark}[thm]{Remark}

\newtheorem{problem}[thm]{Problem}

\newtheorem{thevarthm}[thm]{\varthmname}

\newenvironment{varthm*}[1]{\trivlist\item[]{\bf #1.}\it}{\endtrivlist}

\renewcommand\geq{\geqslant}

\renewcommand\leq{\leqslant}

\newcommand\be{\begin{eqnarray*}}
\newcommand\ee{\end{eqnarray*}}

\newcommand\K{\mathbb K}
\renewcommand\P{\mathbb P}

\newcommand\calh{{\mathcal H}}
\newcommand\calf{{\mathcal F}}
\newcommand\fa[2]{\calf_{#1}^{#2}}
\newcommand\ef{{\mathcal{RF}_3^n(1)}}
\newcommand\newop[2]{\def#1{\mathop{\rm #2}\nolimits}}
\newop\log{log}
\newop\ord{ord}
\newop\Gal{Gal}
\newop\SL{SL}
\newop\Bl{Bl}
\newop\mult{mult}
\newop\imult{imult}
\newop\mass{mass}
\newop\Ass{Ass}
\newop\div{div}
\newop\codim{codim}
\newop\sing{sing}
\newop\Zeroes{Zeroes}

\newcommand\wtilde[1]{\widetilde{#1}}
\newcommand\what[1]{\widehat{#1}}

\def\keywordname{{\bfseries Keywords}}%
\def\keywords#1{\par\addvspace\medskipamount{\rightskip=0pt plus1cm
\def\and{\ifhmode\unskip\nobreak\fi\ $\cdot$
}\noindent\keywordname\enspace\ignorespaces#1\par}}
\def\subclassname{{\bfseries Mathematics Subject Classification
(2000)}\enspace}
\def\subclass#1{\par\addvspace\medskipamount{\rightskip=0pt plus1cm
\def\and{\ifhmode\unskip\nobreak\fi\ $\cdot$
}\noindent\subclassname\ignorespaces#1\par}}

\def\endproof{\hspace*{\fill}\endproofsymbol\endtrivlist}

\def\endproofsymbol{\frame{\rule[0pt]{0pt}{6pt}\rule[0pt]{6pt}{0pt}}}

\begin{document}

\author{Grzegorz Malara, Justyna Szpond}
\title{Fermat-type configurations of lines in $\P^3$ and the containment problem}
\date{}
\maketitle
\thispagestyle{empty}

\begin{abstract}
   The purpose of this note is to show a new series of examples of homogeneous ideals $I$ in
   $\K[x,y,z,w]$ for which the containment $I^{(3)}\subset I^2$ fails. These ideals
   are supported on certain arrangements of lines in $\P^3$, which resemble Fermat configurations
   of points in $\P^2$, see \cite{NagSec16}. All examples exhibiting the failure of the containment
   $I^{(3)}\subseteq I^2$ constructed so far have been supported on points or cones over configurations
   of points. Apart from providing new counterexamples, these ideals seem quite interesting
   on their own.

\keywords{containment problem, symbolic powers, arrangements of hyperplanes, intersection lattice, Fermat configurations}
\subclass{14N20 \and 13F20 \and 13C05 \and 14N05}
\end{abstract}

%*****************************************************************************

\section{Introduction}
We study here containment relations between symbolic and ordinary powers of homogeneous ideals.
\begin{definition}[Symbolic power]
   Let $I\subseteq \K[\P^N]$ be a homogeneous ideal. For $m\geq 0$ the $m$-th symbolic power of $I$ is defined as
   \begin{equation}\label{eq:symbolic power}
      I^{(m)}=\bigcap_{P\in\Ass(I)}\left(I^mR_P\cap R\right),
   \end{equation}
   where $\Ass(I)$ is the set of associated primes of $I$.
\end{definition}

\begin{problem}[The containment problem]
   Let $I\subseteq \K[\P^N]$ be a homogeneous ideal. Determine all pairs of non-negative integers $m$, $r$
   such that there is the containment
   \begin{itemize}
      \item[a)] $I^r\subseteq I^{(m)}$;
      \item[b)] $I^{(m)}\subseteq I^r$.
   \end{itemize}
\end{problem}
Part a) of the Containment Problem has an easy answer: The containment holds if and only if $r\geq m$, see
\cite[Lemma 8.4.1]{PSC}.

Part b) has attracted a lot of attention in recent years. Motivated by Swanson's work on the equivalence of adic
and symbolic topologies on noetherian commutative rings \cite{Swa00},
Ein, Lazarsfeld and Smith proved in \cite{ELS01} a ground-breaking result, which in a form convenient for
this work, reads as follows.
\begin{theorem}[Ein, Lazarsfeld, Smith]\label{ELS}
Let $V\subseteq\P^N$ be a subvariety of codimension $e$ and let $I=I(V)$ be its defining ideal. Then
$$I^{(m)}\subseteq I^r$$
holds for all $m\geq re$.
\end{theorem}
It is natural to wonder to what extent the containment condition in Theorem \ref{ELS} is tight.
This question is meaningful when one asks for \emph{all} ideals with certain properties.
For example, if $I$ is a complete intersection ideal, then $I^{(m)}=I^m$ for all $m\geq 1$ and
thus $I^{(m)}\subseteq I^r$ holds for all $m\geq r$. Prompted by a considerable number of studied examples, the authors in
\cite[Conjecture 8.4.3]{PSC} state the following Conjecture.
\begin{conjecture}\label{conj: codim e}
Let $I$ be a homogeneous ideal with $0\neq I\varsubsetneq \K[\P^N]$ such that $\codim (\Zeroes(I))=e$. Then
$$I^{(m)}\subseteq I^r$$
for $m\geq er-(e-1)$.
\end{conjecture}
If $I$ is an ideal of points in $\P^N$, then $e=N$ and the Conjecture predicts that
$$I^{(m)}\subseteq I^r$$
for all $m\geq Nr-(N-1)$. In particular, for $N=2$ and $r=2$, which is the first non-trivial
case, the Conjecture reduces to the question asked by Huneke around $2000$, if there is the containment
\begin{equation}\label{eq: Huneke}
I^{(3)}\subseteq I^2
\end{equation}
for all ideals of points in $\P^2$.

During the past decade or so, a considerable amount of research has been devoted to the containment
\eqref{eq: Huneke}, and more generally to Conjecture \ref{conj: codim e} in the case $I$ is
an ideal of points in $\P^N$, i.e., $e=N$.

By now, for $e=N$ there are a number of counterexamples available, showing that Conjecture \ref{conj: codim e}
was overoptimistic. In positive characteristic, series of counterexamples involving various $r$
and $N$ have been constructed by Harbourne and Seceleanu in \cite{HarSec15}. In characteristic $0$,
the list of counterexamples is much shorter and they all deal with the containment \eqref{eq: Huneke}
in $\P^2$. The first counterexample was announced by Dumnicki, Szemberg and Tutaj-Gasi\'nska in
\cite{DST13}. Further constructions followed in \cite{Real}, \cite{BCH14}, \cite{HarSec15}, \cite{KKMxx}, \cite{LBM15}, \cite{Sec15}.
See \cite{SzeSzp17} for a survey of counterexamples known until now.

In this paper we show the first non-trivial (i.e. not a cone over points in $\P^2$) counterexample in characteristic $0$ to Conjecture
\ref{conj: codim e} in the case $N=3$ and $e=2$, i.e., for an ideal of a $1$-dimensional
   subscheme in $\P^3$. Our main result is the following.

\begin{varthm*}{Main Theorem}
There exists an arrangement of lines in $\P^3$, not all passing through the same point, such that for its defining ideal $I$
one has
$$I^{(3)}\nsubseteq I^2.$$
\end{varthm*}
In fact, we construct a sequence of arrangements of lines satisfying the Main Theorem.

\section{Arrangement of hyperplanes}
   Arrangements of lines, more precisely point sets defined as intersection points
   of arrangement lines, have played a pivotal role in exhibiting counterexamples to Conjecture \ref{conj: codim e}.
It is therefore  not surprising that arrangements of hypersurfaces and higher dimensional flats in which
they intersect, lead to counterexamples based on higher dimensional subvarieties.
   Of course, there is a trivial construction of a cone over a counterexample configuration
   of points in $\P^2$. A non-trivial example requires
   arrangements to be picked much more carefully. We present here a series of such constructions.

We begin this section with introducing some general notation for arrangements of hyperplanes in projective
spaces. Later on we specialize to what we call Fermat-type arrangements.

\begin{definition}[Intersection lattice]
   Let $\calh=\{H_1,\ldots,H_d\}$ be an arrangement of hyperplanes, i.e., a finite set of codimension $1$
   projective subspaces in $\P^N$. The set $L(\calh)$ of all non-empty intersections of hyperplanes in $\calh$
   is the \emph{intersection lattice of} $\calh$.
\end{definition}
   All elements in $L(\calh)$ are projective flats (i.e. projective subspaces).
   For a flat $V\in L(\calh)$ we define the \emph{multiplicity $\mult_{\calh}(V)$
   of $\calh$ along $V$} as the number of hyperplanes in $\calh$ containing $V$.

   The symbol $\calh(k)$ stands for the union of all
   $k$-dimensional flats in $L(\calh)$.
   In particular $\calh(N-1)$ denotes all hyperplanes in the arrangement, and $\calh(0)$
   denotes all points in $L(\calh)$
   We define the numbers $t_j^{\calh}(k)$ as the number of
   $k$-dimensional flats of multiplicity $j$ in $\calh$.

For $k=N-2$, we have the following fundamental combinatorial equality
\begin{equation}\label{eq:combinatorial codim2}
   \binom{d}{2}=\sum_{j\geq 2}^N \binom{j}{2}\cdot t_j^{\calh}(k).
\end{equation}
For $N=3$ there is a combinatorial formula generalizing \eqref{eq:combinatorial codim2},
due to Hunt, see \cite[part II, Section 6.1.1]{Hunt}.
We denote by $t_{pq}^{\calh}$ the number of points of multiplicity $p$ on $q$-fold lines.
Then we have
\begin{equation}\label{eq:Hunt}
   \binom{d}{3}=\sum_{p\geq 3}t_p^{\calh}(0)\binom{p}{3}-\sum_{q\geq 3}\left(\sum_{p\geq q}t_{pq}^{\calh}-t_q^{\calh}(1)\right)\binom{q}{3}.
\end{equation}

\section{Fermat configurations in $\P^2$}
The Fermat arrangement $\fa{2}{n}$ of lines in $\P^2$ for a given $n$ is defined by the equation
$$(x^n-y^n)(y^n-z^n)(z^n-x^n)=0.$$
Thus there are $3n$ lines, which intersect by three in $n^2$ points and there are $3$ additional points
of multiplicity $n$. The non-zero numerical invariants are therefore
$$t_3^{\fa{2}{n}}(0)=n^2  \mbox{ and } t_n^{\fa{2}{n}}(0)=3,$$
see also \cite[Example II.6]{Urz08} for more details.

These arrangements appear in the literature under several names. In the book of Barthel, Hirzebruch and H\"ofer
they are called Ceva arrangements, see \cite[Section 2.3.I]{BHH}. The same terminology is kept in the
   recent book by Tretkoff \cite[Chapter V, 5.2]{Tre16}. Both these beautiful books are focused on
   surfaces of general type arising as ball quotients. In the area of commutative algebra the name
   Fermat arrangement seems more customary, see \cite{NagSec16}.

For $n=3$ we get as a special case the celebrated dual Hesse arrangement with $t_3^{\fa{2}{3}}(0)=12$
and all other invariants vanishing. In the context of the containment problem for
the set of all intersection points of $\fa{2}{n}$ defined by the ideal
$$I_n=(x(y^n-z^n),y(z^n-x^n),z(x^n-y^n))$$
the containment $I_n^{(3)}\subseteq I_n^2$ fails as proved by Harbourne and Seceleanu in \cite[Proposition 2.1]{HarSec15}. A completely different
proof based on a more general theoretical framework has been presented by Seceleanu, see \cite[Proposition 4.2]{Sec15}.
Additionally, ideal-theoretic properties of these configurations have been treated by Nagel and Seceleanu in
\cite{NagSec16}.

\subsection{Fermat configurations in $\P^3$}
In this section we study Fermat arrangements of flats (planes and lines) in $\P^3$.
   Arrangements of this kind are well-known in the literature as they come from
   finite unitary reflection groups, $G(n,n,N)$, see \cite[page 295]{SheTod}. Here we study
   the case $N=3$ and postpone the general case to the forthcoming paper \cite{MSS}.

The Fermat arrangement $\fa{3}{n}$ of planes in $\P^3$ is defined by the vanishing of the polynomial
\begin{equation}\label{eq:poly f}
   F_n(x,y,z,w)=(x^n-y^n)(x^n-z^n)(x^n-w^n)(y^n-z^n)(y^n-w^n)(z^n-w^n).
\end{equation}
There are $6n$ planes in the arrangement. They intersect in triples along $4n^2$ lines and there are $6$
additional lines of multiplicity $n$. Therefore the non-zero invariants of the arrangements $\fa{3}{n}(1)$ are
$$t_2^{\fa{3}{n}}(1)=3n^2,\;\; t_3^{\fa{3}{n}}(1)=4n^2 \; \mbox{ and } \; t_n^{\fa{3}{n}}(1)=6.$$
The six $n$--fold lines of the arrangement are the edges of the coordinate tetrahedron.
   Since we are interested in the third symbolic power of an ideal,
   we restrict our attention to those lines in $\fa{3}{n}(1)$ which have multiplicity at least $3$.
\begin{definition}[The restricted Fermat configuration of lines]
   The \emph{restricted Fermat configuration} $\ef$ of lines in $\P^3$
   is the union of all lines in $\fa{3}{n}(1)$ with multiplicity at least $3$.
\end{definition}

For completeness we list also non-zero dimensional invariants of $\fa{3}{n}(0)$. Using \eqref{eq:Hunt} we have
$$t_6^{\fa{3}{n}}(0)=n^3,\;\;t_{n+1}^{\fa{3}{n}}(0)=6n\;\mbox{  and }\; t_{3n}^{\fa{3}{n}}(0)=4.$$
The four $3n$-fold points in the arrangement are the vertices of the coordinate tetrahedron.

\section{The non-containment result}
We begin by an explicit description of the ideal $I_n$ defining the $6+4n^2$ lines in the restricted Fermat arrangement $\ef$ in $\P^3$.
\begin{lemma}\label{lem: ideal of lines}
The ideal $I_n=I(\ef)$ is generated by
\[
\begin{array}{cc}
g_1=(x^n-y^n)(z^n-w^n)xy, \; & g_2=(x^n-y^n)(z^n-w^n)zw,\\
g_3=(x^n-z^n)(y^n-w^n)xz, \; & g_4=(x^n-z^n)(y^n-w^n)yw,\\
g_5=(x^n-w^n)(y^n-z^n)xw, \; & g_6=(x^n-w^n)(y^n-z^n)yz.
\end{array}
\]
\end{lemma}
\begin{proof}
   The ideal $J_n$ of the $4n^2$ triple lines is a complete intersection ideal
   $$J_n=((x^n-y^n)(z^n-w^n), (x^n-z^n)(y^n-w^n)).$$
   Since
   $$(x^n-w^n)(y^n-z^n)=(x^n-z^n)(y^n-w^n)-(x^n-y^n)(z^n-w^n),$$
   it is clear that for every line in the set of lines defined by $J_n$,
   there are three planes from $\calf_3^n(2)$ vanishing along this line.

   Furthermore, it is clear that, for example, any plane defined by the vanishing
   of $(x^n-y^n)$ belongs to the pencil of planes vanishing along the line $(x,y)$.
   Thus there are $n$ planes in the arrangement, which vanish along the $(x,y)$ line.
   The same holds, by symmetry, for an arbitrary pair of variables, i.e., any
   other coordinate line.

   The ideal of the union of $4n^2$ triple lines and the $6$ coordinate lines is then defined by
   $$I_n=J_n\cap(x,y)\cap(x,z)\cap(x,w)\cap(y,z)\cap(y,w)\cap(z,w).$$
   The generators of $I_n$ can then be easily read off of this presentation.
\end{proof}
   From the Zariski-Nagata Theorem, see \cite[Theorem 3.14]{Eisenbud} we obtain immediately
\begin{corollary}\label{cor:f in I3}
   The polynomial $F_n$ in \eqref{eq:poly f} is an element of $I_n^{(3)}$.
\end{corollary}
   The polynomial $F_n$ is exactly the reason for the non-containment in the Main Theorem.
   We state now a Theorem from which the Main Theorem follows immediately.
\begin{theorem}\label{thm:main true}
   For an arbitrary integer $n\geq 3$ and $I_n$ the ideal of lines
   in the restricted Fermat arrangement $\ef$ the containment
   $$I_n^{(3)}\subset I_n^2$$
   fails.
\end{theorem}
\proof
   We know from Corollary \ref{cor:f in I3} that the polynomial $F_n(x,y,z,w)$
   is contained in $I_n^{(3)}$. We will show that it is not contained in $I_n^2$.
   To this end, we parrot to some extent the proof for the dual Hesse arrangement from \cite[Theorem 2.2]{DST13}
   and for Fermat arrangements $\fa{2}{n}$ from \cite[Proposition 2.1]{HarSec15}.

   Keeping the notation from Lemma \ref{lem: ideal of lines}, we assume
   to the contrary that $F_n\in I_n^2$. Then there are homogeneous polynomials $h_{i,j}$ for $1\leq i\leq j\leq 6$
   of degree $2n-4$ (this count is the reason for the assumption $n\geq 3$) such that
   \begin{equation}\label{eq:f in I2}
      F_n=\sum_{1\leq i\leq j\leq 6}h_{i,j}g_ig_j.
   \end{equation}
   Taking the identity \eqref{eq:f in I2} modulo $(x)$ we obtain
   %\begin{equation}
   \begin{eqnarray}\label{eq:mod x}
      -y^nz^nw^n(y^n-z^n)(y^n-w^n)(z^n-w^n) & =  y^{2n}z^2w^2(z^n-w^n)^2\wtilde{h}_{2,2} \nonumber\\
      & + z^{2n}y^2w^2(y^n-w^n)^2\wtilde{h}_{4,4} \nonumber\\
      & + w^{2n}y^2z^2(y^n-z^n)^2\wtilde{h}_{6,6} \nonumber\\
      & + y^{n+1}z^{n+1}w^2(y^n-w^n)(z^n-w^n)\wtilde{h}_{2,4} \nonumber\\
      & + y^{n+1}w^{n+1}z^2(y^n-z^n)(z^n-w^n)\wtilde{h}_{2,6} \nonumber\\
      & + z^{n+1}w^{n+1}y^2(y^n-z^n)(y^n-w^n)\wtilde{h}_{4,6}.\nonumber \\
   \end{eqnarray}%\end{equation}
   We write here $\wtilde{f}$ to indicate the residue class of a polynomial $f\in\K[x,y,z,w]$
   modulo $(x)$.

   Now we look at the coefficient of the monomial $y^{3n}z^{2n}w^n$ in \eqref{eq:mod x}.
   On the left hand side of the equation this coefficient is $-1$. On the right hand side
   of the equation, the monomial $y^{3n}z^{2n}w^n$ can be obtained only from the second
   summand as a summand in the product
   $$y^{2n+2}z^{2n}w^2\cdot \wtilde{h}_{4,4}.$$
   Taking this for granted for a while, this shows that the coefficient of the monomial $y^{n-2}w^{n-2}$ in $\wtilde{h}_{4,4}$,
   and hence also in $h_{4,4}$ is $-1$.

   Turning to the occurrence of the monomial $y^{3n}z^{2n}w^n$ on the right hand side of \eqref{eq:mod x},
   note to begin with that its coefficient must be zero in the first summand. Indeed, already in the
   product $y^{2n}z^2w^2(z^n-w^n)^2$ either the power of $z$ or the power of $w$ is too large, i.e.,
   exceeds corresponding powers of $z$ and $w$ in $y^{3n}z^{2n}w^n$. Similarly, in the third, fifth and sixth
   summand the powers of $w$ are too large, whereas in the fourth summand either the power of $z$ or that of $w$
   are too large.

   The idea now is to consider \eqref{eq:f in I2} modulo $(z)$ and identify the coefficient of the monomial
   $y^{n-2}w^{n-2}$ in $h_{4,4}$ as $1$, which gives clearly a contradiction.

   Turning to the details, \eqref{eq:f in I2} modulo $(z)$ gives
   %\begin{equation}
   \begin{eqnarray}\label{eq:mod z}
      -x^ny^nw^n(x^n-y^n)(x^n-w^n)(y^n-w^n)
      & = & w^{2n}x^2y^2(x^n-y^n)^2\what{h}_{1,1} \nonumber\\
      & + & x^{2n}y^2w^2(y^n-w^n)^2\what{h}_{4,4} \nonumber\\
      & + & y^{2n}x^2w^2(x^n-w^n)^2\what{h}_{5,5} \nonumber\\
      & - & x^{n+1}w^{n+1}y^2(x^n-y^n)(y^n-w^n)\what{h}_{1,4} \nonumber\\
      & - & y^{n+1}w^{n+1}x^2(x^n-y^n)(x^n-w^n)\what{h}_{1,5} \nonumber\\
      & + & x^{n+1}y^{n+1}w^2(x^n-w^n)(y^n-w^n)\what{h}_{4,5}.\nonumber\\
   \end{eqnarray}
   %\end{equation}
   We write now $\what{f}$ to indicate the residue class of a polynomial $f\in\K[x,y,z,w]$
   modulo $(z)$. This time it is the monomial $x^{2n}y^{3n}w^n$ which appears with the coefficient
   $+1$ on the left hand side of \eqref{eq:mod z} and which can appear only in the second
   summand on the right hand side of \eqref{eq:mod z}. The argument is exactly the same
   as for the modulo $(x)$ case, with necessary adjustment of used variables. The conclusion
   now is that the coefficient of the monomial $y^{n-2}w^{n-2}$ of $h_{4,4}$ is $+1$.
   This contradiction shows that \eqref{eq:f in I2} cannot hold, which in turns means
   that $F_n\notin I_n^2$ as asserted.
\endproof
\section{Higher dimensional generalizations}
There is no reason to restrict this construction to $\P^3$. In $\P^N$, we define degree $n$ Fermat
arrangement $\fa{N}{n}$ of hyperplanes as given by the zero-locus of the polynomial
$$F_n(x_0,\ldots,x_N)=\prod_{0\leq i<j\leq N}(x_i^n-x_j^n).$$
\begin{remark}
   For $n=1$ we obtain the well-known braid arrangement \cite[Example 1.3]{Sta04}. From this point of view, our construction can be
   also viewed as higher order braid arrangements.
\end{remark}
   In \cite{MS17} we show that in codimension $2$ for $I=I(N,n)=I({\mathcal R}\fa{N}{n}(N-2))$, the containment
   $$I^{(3)}\subseteq I^2$$
   still fails.
   On the other hand, there are several interesting algebraic properties of ideals
   $I(\fa{N}{n}(k))$, which we study in the forthcoming article \cite{MSS}.

\paragraph*{Acknowledgement.}
   Research of Szpond was partially supported by National Science Centre, Poland, grant
   2014/15/B/ST1/02197.
   Research of Malara was partially supported by National Science Centre, Poland, grant 2016/21/N/ST1/01491.
   We would like to thank Piotr Pokora and Tomasz Szemberg for helpful conversations.
   Many thanks go to the referee for patience and valuable comments on the original draft of the paper.
%*****************************************************************************

%***************************************************************************** % Addresses

\bigskip \small

   Grzegorz Malara,
   Department of Mathematics, Pedagogical University of Cracow,
   Podchor\c a\.zych 2,
   PL-30-084 Krak\'ow, Poland.

\nopagebreak
   \textit{E-mail address:} \texttt{grzegorzmalara@gmail.com}

\bigskip
   Justyna Szpond,
   Department of Mathematics, Pedagogical University of Cracow,
   Podchor\c a\.zych 2,
   PL-30-084 Krak\'ow, Poland.

\nopagebreak
   \textit{E-mail address:} \texttt{szpond@gmail.com}

%*****************************************************************************

\end{document}